\documentclass[a4paper, 12pt]{article} 
\usepackage{amssymb,latexsym} \author{
\small Centre de Math\'ematiques Laurent Schwartz\\
\small Ecole Polytechnique\\
\small FR-91128 Palaiseau C\'edex\\
\footnotesize johannes@math.polytechnique\\
\footnotesize  and UMR 7640, CNRS
} \date{}
\title{Almost sure Weyl asymptotics for non-self-adjoint elliptic
  operators on compact manifolds}
\author{William Bordeaux Montrieux\footnote{Fakult\"at f\"ur Mathematik, Universit\"at Wien, Nordbergstrasse 15, 1090 Wien,
Austria,
william.bordeaux-montrieux@univie.ac.at}
\and
Johannes Sj\"ostrand\footnote{IMB, Universit\'e de Bourgogne, 9,
Av. A. Savary, BP 47870, FR 21078 Dijon c\'edex, France, 
and UMR 5584 CNRS,
johannes.sjostrand@u-bourgogne.fr}\ \footnote{Ce travail a b\'en\'efici\'e d'une aide de l'Agence Nationale de la Recherche
portant la r\'ef\'erence ANR-08-BLAN-0228-01}}

\newtheorem{dref}{Definition}[section] 
\newtheorem{theo}[dref]{Theorem} \newtheorem{prop}[dref]{Proposition}

\newenvironment{proof}{\par\noindent{{\bf Proof.}}}{\hfill$\Box$
\medskip} 
\newenvironment{proofof}{\par\noindent{{\bf Proof of}}}{\hfill$\Box$
\medskip}

\newcommand{\ekv}[2]{\begin{equation}\label{#1}#2\end{equation}}
\newcommand{\eekv}[3]{\begin{eqnarray}\label{#1}#2 \\ #3
\nonumber\end{eqnarray}}

\begin{document}

\maketitle
\begin{abstract} In this paper, we consider elliptic differential
  operators on compact manifolds with a random perturbation in the 0th
  order term and show under fairly weak additional assumptions that
  the large eigenvalues almost surely distribute according to the
  Weyl law, well-known in the self-adjoint case.

\medskip \par \centerline{\bf R\'esum\'e} Dans ce travail, nos consid\'erons des op\'erateurs 
diff\'erentiels
elliptiques sur des vari\'et\'es compactes avec une perturbation
al\'eatoire dans le terme d'orde 0.
Sous des hypoth\`eses suppl\'ementaires assez faibles, nous montrons
que les grandes valeurs propres se distribuent selon la loi de Weyl,
bien connue dans le cas auto-adjoint.

\end{abstract}

\tableofcontents

\section{Introduction}\label{int}
\setcounter{equation}{0}

This work is a continuation of a series of works concerning the
asymptotic distribution of eigenvalues for non-self-adjoint
(pseudo-)differential operators with random perturbations. Since the
works of L.N.~Trefethen \cite{Tr}, E.B.~Davies \cite{Da}, M.~Zworski \cite{Zw} and many others (see
for instance \cite{Ha} for further references) we
know that the resolvents of such operators tend to have very large
norms when the spectral parameter is in the range of the symbol, and
consequently, the eigenvalues are unstable under small perturbations
of the operator. It is therefore quite natural to study the effect of 
random perturbations. Mildred Hager
\cite{Ha} studied quite general classes of non-self-adjoint
$h$-pseudodifferential operators on the real line with a suitable random
potential added, and she showed that the eigenvalues distribute
according to the natural Weyl law with a probability very close to 1
in the semi-classical limit ($h\to 0$). Due to the method, this 
result was restricted to the interior of the range 
of the leading symbol $p$ of the
operator and with a non-vanishing assumption on the Poisson bracket $\{
p,\overline{p}\}$. 

In \cite{HaSj} the results were generalized to higher dimension and 
the boundary of the range of $p$ could be included, but the
perturbations where no more multiplicative. In \cite{Sj1, Sj2} further
improvements of the method were introduced and the case of
multiplicative perturbations was handled in all dimensions. 

W.~Bordeaux Montrieux \cite{Bo} studied
elliptic systems of differential operators on $S^1$
with random perturbations of the coefficients,
and under some additional assumptions, he showed that the large eigenvalues obey the
Weyl law \emph{almost surely}. His analysis was based on a reduction to
the semi-classical case (using essentially the Borel-Cantelli lemma), 
where he could use and extend the methods of Hager \cite{Ha}. 

The purpose of the present work is to extend the results of \cite{Bo}
to the case of elliptic operators on compact manifolds by replacing
the one dimensional semi-classical techniques by the more recent
result of \cite{Sj2}. For simplicity, we treat only the scalar case
and the random perturbation is a potential.

Let $X$ be a smooth compact manifold of dimension $n$. Let $P^0$ be an
elliptic differential operator on $X$ of order $m\ge 2$ with smooth
coefficients and with
principal symbol $p(x,\xi )$. In local coordinates we get, using
standard multi-index notation,
\ekv{in.1}
{
P^0=\sum_{|\alpha |\le m}a_\alpha ^0(x)D^\alpha ,\quad 
p(x,\xi )=\sum_{|\alpha |= m}a_\alpha ^0(x)\xi ^\alpha.
}
Recall that the ellipticity of $P^0$ means that $p(x,\xi )\ne 0$ for
$\xi \ne 0$. We assume that
\ekv{in.2}
{
p(T^*X)\ne {\bf C}.
}
Fix a strictly positive smooth density of integration $dx$ on $X$, so
that the $L^2$ norm $\Vert \cdot \Vert$ and inner product $(\cdot
|\cdot \cdot )$ are unambiguously defined. Let $\Gamma :L^2(X)\to
L^2(X)$ be the antilinear operator of complex conjugation, given by
$\Gamma u=\overline{u}$. We need the symmetry assumption
\ekv{in.3}
{
P^*=\Gamma P\Gamma ,
}
where $P^*$ is the formal complex adjoint of $P$. As in \cite{Sj2} we
observe that the property (\ref{in.3}) implies that
\ekv{in.4}
{
p(x,-\xi )=p(x,\xi ),
}
and conversely, if (\ref{in.4}) holds, then the operator
$\frac{1}{2}(P+\Gamma P\Gamma )$ has the same principal symbol $p$ and
satisfies (\ref{in.3}).

\par Let $\widetilde{R}$ be an elliptic differential operator on $X$
with smooth coefficients, which is self-adjoint and strictly
positive. Let $\epsilon _0,\epsilon _1,...$ be an orthonormal basis of
eigenfunctions of $\widetilde{R}$ so that 
\ekv{in.5}
{
\widetilde{R}\epsilon _j=(\mu _j^0)^2\epsilon _j,\quad 0<\mu _0^0<\mu
_1^0\le \mu _2^0\le ...
}
Our randomly perturbed operator is 
\ekv{in.6}
{
P_\omega ^0=P+q_\omega ^0(x),
}
where $\omega $ is the random parameter and 
\ekv{in.7}
{
q_\omega ^0(x)=\sum_{0}^\infty \alpha _j^0(\omega )\epsilon _j.
}
Here we assume that $\alpha _j^0(\omega )$ are independent complex
Gaussian random variables of variance $\sigma _j^2$ and mean value 0:
\ekv{in.8}
{
\alpha _j^0\sim {\cal N}(0,\sigma _j^2),
}
where 
\ekv{in.8.5}
{
(\mu _j^0)^{-\rho }e^{-(\mu _j^0)^{\frac{\beta }{M+1}}}\lesssim 
\sigma _j\lesssim (\mu _j^0)^{-\rho },
}
\ekv{in.9}
{M=\frac{3n-\frac{1}{2}}{s-\frac{n}{2}-\epsilon },\ 0\le \beta
  <\frac{1}{2},\ 
\rho >n,
}
where $s$, $\rho $, $\epsilon $ are fixed constants such that
$$
\frac{n}{2}<s<\rho -\frac{n}{2},\ 0<\epsilon <s-\frac{n}{2}.
$$

\par Let $H^s(X)$ be the standard Sobolev space of order $s$. As will
follow from considerations below, we have $q_\omega^0 \in H^s(X)$ almost
surely since $s<\rho -\frac{n}{2}$. Hence $q_\omega^0\in L^\infty $
almost surely, implying that $P_\omega ^0$ has purely discrete
spectrum.

\par Consider the function $F(\omega )=\mathrm{arg\,}p(\omega )$ on
$S^*X$. For given $\theta _0\in S^1\simeq {\bf R}/(2\pi {\bf Z})$,
$N_0\in \dot{{\bf N}}:={\bf N}\setminus \{ 0\}$, we introduce the property $P(\theta _0,N_0)$:
\ekv{in.10}
{
\sum_1^{N_0}|\nabla ^kF(\omega )|\ne 0\hbox{ on }\{ \omega \in S^*X;\,
F(\omega )=\theta _0\}.
}
Notice that if $P(\theta _0,N_0)$ holds, then $P(\theta ,N_0)$ holds
for all $\theta $ in some neighborhood of $\theta _0$.

\par We can now state our main result.

\begin{theo}\label{in1}
Assume that $m\ge 2$. Let $0\le \theta _1\le \theta _2\le 2\pi $ and
assume that $P(\theta _1,N_0)$ and $P(\theta _2,N_0)$ hold for some
$N_0\in\dot{{\bf N}}$. Let $g\in C^\infty ([\theta _1,\theta
_2];]0,\infty [)$ and put 
$$
\Gamma ^g_{\theta _1,\theta _2;0,\lambda }=\{ re^{i\theta } ; \theta
  _1\le \theta \le \theta _2,\ 0\le r\le \lambda g(\theta )\}.
$$
Then for every $\delta \in ]0,\frac{1}{2}-\beta [$ there exists $C>0$ such
that almost surely: $\exists C(\omega )<\infty $ such that for all
$\lambda \in [1,\infty [$:
\eekv{in.11}
{
|\#(\sigma (P_\omega ^0)\cap \Gamma _{\theta _1,\theta _2;0,\lambda }^g)
-\frac{1}{(2\pi )^n}\mathrm{vol\,}p^{-1}(\Gamma ^g_{\theta _1,\theta
  _2;0,\lambda })
|
}
{\le C(\omega )+C\lambda ^{\frac{n}{m}-\frac{1}{m}(\frac{1}{2}-\beta -\delta
    )\frac{1}{N_0+1}}.}
Here $\sigma (P_\omega ^0)$ denotes the spectrum and $\# (A)$ denotes
the number of elements in the set $A$. In (\ref{in.11}) the
eigenvalues are counted with their algebraic multiplicity.
\end{theo}

The proof actually allows to have almost surely a simultaneous
conclusion for a whole family of $\theta _1,\theta _2,g$:

\begin{theo}\label{in2}
Assume that $m\ge 2$. Let $\Theta $ be a compact subset of $[0,2\pi ]$. Let
$N_0\in {\bf N}$ and assume that $P(\theta ,N_0)$ holds uniformly for
$\theta \in \Theta $. Let ${\cal G}$ be a subset of $\{(g,\theta
_1,\theta _2);\ \theta _j\in \Theta, \theta _1\le \theta _2,\ g\in 
C^\infty ([\theta _1,\theta
_2];]0,\infty [)
\}$ with the property that $g$ and $1/g$ are uniformly bounded in $C^\infty ([\theta _1,\theta
_2];]0,\infty [)$ when $(g,\theta _1,\theta _2)$ varies in ${\cal
  G}$. Then for every $\delta \in ]0,\frac{1}{2}-\beta [$ 
there exists $C>0$ such
that almost surely: $\exists C(\omega )<\infty $ such that for all
$\lambda \in [1,\infty [$ and all $(g,\theta _1,\theta _2)\in {\cal
  G}$, we have the estimate (\ref{in.11}).
\end{theo}

The condition (\ref{in.8.5}) allows us to choose $\sigma _j$
 decaying faster than any negative power of
$\mu _j^0$. Then from the
discussion below, it will follow that $q_\omega (x)$ is almost
surely a smooth function. A rough and somewhat intuitive
interpretation of Theorem \ref{in2} is then that for 
almost every elliptic
operator of order $\ge 2$ with smooth coefficients on a compact manifold which satisfies the 
conditions (\ref{in.2}), (\ref{in.3}), the large eigenvalues
distribute according to Weyl's law in sectors with limiting directions
that satisfy a weak non-degeneracy condition.

\section{Volume considerations}\label{vo}
\setcounter{equation}{0}

In the next section we shall perform a reduction to a semi-classical
situation and work with $h^mP_0$ which has the semi-classical
principal symbol $p$ in (\ref{in.1}). As in \cite{HaSj, Sj1, Sj2}, we introduce
\ekv{vo.1}
{
V_z(t)=\mathrm{vol\,}\{ \rho \in T^*X;\, |p(\rho )-z|^2\le t\},\ t\ge 0.
}
\begin{prop}\label{vo1}
For any compact set $K\subset \dot{{\bf C}}={\bf C}\setminus \{ 0\}$, we have
\ekv{vo.2}
{
V_z(t)={\cal O}(t^\kappa ),\hbox{ uniformly for }z\in K,\ 0\le t\ll 1,
}
with $\kappa =1/2$.
\end{prop}

The property (\ref{vo.2}) for some $\kappa \in ]0,1[$ is required in
\cite{HaSj, Sj1, Sj2} near the boundary of the set $\Gamma $, where we count
the eigenvalues. Another important quantity appearing there was 
\ekv{vo.3}
{
\mathrm{vol\,}(\gamma +D(0,t)),
} 
where $\gamma =\partial \Gamma $ and   $\Gamma \Subset \dot{{\bf C}}$ is assumed to have piecewise 
smooth boundary. From (\ref{vo.2}) with  general $\kappa $ it follows
that the volume (\ref{vo.3}) is ${\cal O}(t^{2\kappa -1})$, which is
of interest when $\kappa >1/2$. In our case, we shall therefore
investigate $\mathrm{vol\,}(\gamma +B(0,t))$ more directly, when
$\gamma $ is (the image of) a smooth curve. The following result
implies Proposition \ref{vo1}:
\begin{prop}\label{vo2}
Let $\gamma $ be the curve $\{ re^{i\theta }\in {\bf C};\, r=g(\theta
),\ \theta \in S^1 \}$, where $0<g\in C^1(S^1)$. Then 
$$
\mathrm{vol\,}(p^{-1}(\gamma +D(0,t)))={\cal O}(t),\ t\to 0.
$$
\end{prop}
\begin{proof}
This follows from the fact that the radial derivative of $p$ is $\ne
0$. More precisely, write $T^*X\setminus 0\ni \rho =r\omega $, $\omega
\in S^*X$, $r>0$, so that $p(\rho )=r^mp(\omega )$, $p(\omega )\ne
0$. If $\rho \in p^{-1}(\gamma +D(0,t))$, we have for some $C\ge 1$,
independent of $t$,
$$
g(\mathrm{arg\,}p(\omega ))-Ct\le r^m|p(\omega )|\le
g(\mathrm{arg\,}p(\omega ))+Ct,
$$
$$
\left( \frac{g(\mathrm{arg\,}p(\omega ))-Ct}{|p(\omega )|}
\right)^{\frac{1}{m}} \le r\le     \left( \frac{g(\mathrm{arg\,}p(\omega ))+Ct}{|p(\omega )|}
\right)^{\frac{1}{m}},
$$
so for every $\omega \in S^*X$, $r$ has to belong to an interval of
length ${\cal O}(t)$.
\end{proof}

\par We next study the volume in (\ref{vo.3}) when $\gamma $ is a
radial segment of the form $[r_1,r_2]e^{i\theta _0}$, where
$0<r_1<r_2$ and $\theta _0\in S^1$.
\begin{prop}\label{vo3}
Let $\theta _0\in S^1$, $N_0\in\dot{{\bf N}}$ and assume that
$P(\theta _0,N_0)$ holds. Then if $0<r_1<r_2$ and $\gamma $ 
is the radial segment $[r_1,r_2]e^{i\theta _0}$, we have 
$$
\mathrm{vol\,}(p^{-1}(\gamma +D(0,t)))={\cal O}(t^{1/N_0}),\ t\to 0.
$$
\end{prop}
\begin{proof}
We first observe that it suffices to show that
$$
\mathrm{vol}_{S^*X}F^{-1}([\theta _0-t,\theta _0+t])={\cal O}(t^{1/N_0}).
$$
This in turn follows for instance from the Malgrange preparation
theorem: At every point $\omega _0\in F^{-1}(\theta _0)$ we can choose
coordinates $\omega _1,...,\omega _{2n-1}$, centered at $\omega _0$,
such that for some $k\in\{ 1,...,N_0\}$, we have that $\partial _{\omega
  _1}^j(F-\theta _0)(\omega _0)$ is $=0$ when $0\le j\le k-1$ and $\ne
0$ when $j=k$. Then by Malgrange's preparation theorem, we have
$$
F(\omega )-\theta _0=G(\omega )(\omega _1^k+a_1(\omega _2,...,\omega
_{2n-1})\omega _1^{k-1}+...+a_k(\omega _2,...,\omega _{2n-1})),
$$
where $G,a_j$ are real and smooth, $G(\omega _0)\ne 0$,
and it follows that 
$$
\mathrm{vol\,}(F^{-1}([\theta _0-t,\theta
_0+t])\cap\mathrm{neigh\,}(\omega _0))={\cal O}(t^{1/k}).
$$
It then suffices to use a simple compactness argument.
\end{proof}

\par Now, let $0\le \theta _1<\theta _2\le 2\pi $, $g\in C^\infty
([\theta _1,\theta _2];]0,\infty [)$ and put 
\ekv{vo.4}
{
\Gamma ^g_{\theta _1,\theta _2;r_1,r_2}=\{ re^{i\theta };\, \theta
_1\le \theta \le \theta _2,\ r_1g(\theta )\le r\le r_2g(\theta )\} ,
}
for $0\le r_1\le r_2 <\infty $. If $0<r_1<r_2<+\infty $ and
$P(\theta _j,N_0)$ hold for $j=1,2$, then the last two propositions
imply that 
\ekv{vo.5}
{
\mathrm{vol\,}p^{-1}(\partial \Gamma ^g_{\theta _1,\theta
  _2;r_1,r_2}+D(0,t))={\cal O}(t^{1/N_0}),\ t\to 0.
}

\section{Semiclassical reduction}\label{sc}
\setcounter{equation}{0}

We are interested in the distribution of large eigenvalues $\zeta $ of
$P_\omega ^0$, so we make a standard reduction to a semi-classical
problem by letting $0<h\ll 1$ satisfy
\ekv{sc.1}
{
\zeta =\frac{z}{h^m},\ |z|\asymp 1,\ h\asymp |\zeta |^{-1/m},
}
and write
\ekv{sc.2}
{
h^m(P_\omega ^0-\zeta )=h^mP_\omega ^0-z=:P+h^mq^0_\omega -z,
}
where
\ekv{sc.3}
{
P=h^mP^0=\sum_{|\alpha |\le m} a_\alpha (x;h)(hD)^\alpha .
}
Here 
\eekv{sc.4}
{a_\alpha (x;h)&=&{\cal O}(h^{m-|\alpha |})\hbox{ in }C^\infty ,}
{a_\alpha (x;h)&=&a_\alpha ^0(x)\hbox{ when }|\alpha |=m.}
So $P$ is a standard semi-classical differential operator with
semi-classical principal symbol $p(x,\xi )$.

Our strategy will be to decompose the random perturbation
$$
h^mq_\omega ^0=\delta Q_\omega +k_\omega (x),
$$
where the two terms are independent, 
and with probability very close to 1, $\delta Q_\omega $ will be a
semi-classical random perturbation as in \cite{Sj2} while
\ekv{sc.5}
{
\Vert k_\omega \Vert_{H^s}\le h,
}
and 
\ekv{sc.5.5}{
s\in ]\frac{n}{2},\rho -\frac{n}{2}[} is fixed. Then $h^mP_\omega
^0$ will be viewed as a random perturbation of $h^mP^0+k_\omega $.
In order to achieve this without
extra assumptions on the order $m$, we will also have to represent
some of our eigenvalues $\alpha _j^0(\omega )$ as sums of two
independent Gaussian random variables.

\par We start by examining when
\ekv{sc.6}
{
\Vert h^mq_\omega ^0\Vert_{H^s}\le h.
}
\begin{prop}\label{sc1}
There is a constant $C>0$ such that (\ref{sc.6}) holds with
probability
$$
\ge 1-\exp (C-\frac{1}{2Ch^{2(m-1)}}).
$$
\end{prop}
\begin{proof}
We have
\ekv{sc.7}
{
h^mq_\omega ^0=\sum_0^\infty \alpha _j(\omega )\epsilon _j,\quad
\alpha _j=h^m\alpha _j^0\sim {\cal N}(0,(h^m\sigma _j)^2),
}
and the $\alpha _j$ are independent. Now, using standard functional
calculus for $\widetilde{R}$ as in \cite{Sj1, Sj2}, we see that
\ekv{sc.8}
{
\Vert h^mq_\omega ^0\Vert_{H^s}^2\asymp \sum_0^\infty \vert (\mu
_j^0)^s\alpha _j(\omega )\vert ^2,
}
where $(\mu _j^0)^s\alpha _j\sim {\cal N}(0,(\widetilde{\sigma }_j)^2)$ are
independent random variables and $\widetilde{\sigma }_j=(\mu
_j^0)^sh^m\sigma _j$.

\par Now recall the following fact, established by Bordeaux Montrieux
\cite{Bo}, improving and simplifying a similar result in \cite{HaSj}:
Let $d_0,d_1,...$ be a finite or infinite family of independent
complex Gaussian random variables, $d_j\sim {\cal N}(0,(\widehat{\sigma
}_j)^2)$, $0<\widehat{\sigma }_j<\infty $, and assume that $\sum
\widehat{\sigma }_j^2<\infty $. Then for every $t>0$,
\ekv{sc.9}
{
{\bf P}(\sum\vert d_j\vert^2\ge t)\le \exp (\frac{-1}{2\max
  \widehat{\sigma }_j^2}(t-C_0\sum \widehat{\sigma }_j^2)).
}
 Here ${\bf P}(A)$ denotes the
probability of the event $A$ and $C_0>0$ is a universal
constant. The estimate is interesting only when $t>C_0\sum
\widehat{\sigma }_j^2$ and for such values of $t$ it improves
if we replace $\{ d_0,d_1,...\}$ by a subfamily. Indeed,
$\sum \widehat{\sigma }_j^2$ will then decrease and so will $\max
\widehat{\sigma }_j^2$. 

Apply this to (\ref{sc.8}) with $d_j=(\mu _j^0)^s \alpha _j$,
$t=h^2$. 
Here, we recall that $\widetilde{\sigma }_j=(\mu _j^0)^s h^m\sigma
_j$, and get from (\ref{in.8.5}), (\ref{sc.5.5}) that  
\ekv{sc.10}
{
\max \widetilde{\sigma }_j^2\asymp  h^{2m},
}
while
\ekv{sc.11}
{
\sum_0^\infty \widetilde{\sigma }_j^2\lesssim h^{2m}\sum_0^\infty (\mu
_j^0)^{2(s-\rho )}.
}
Let $N(\mu )=\# (\sigma (\sqrt{\widetilde{R}})\cap ]0,\mu ])$ be the
number of eigenvalues of $\sqrt{\widetilde{R}}$ in $]0,\mu ]$, so that
$N(\mu )\asymp \mu ^n$ by the standard Weyl asymptotics for positive
elliptic operators on compact manifolds. The last sum in (\ref{sc.11})
is equal to 
$$
\int_0^\infty \mu ^{2(s-\rho )}dN(\mu )=\int_0^\infty 2(\rho -s)\mu
^{2(s-\rho )-1}N(\mu )d\mu ,
$$
which is finite since $2(s-\rho )+n<0$ by (\ref{sc.5.5}). Thus 
\ekv{sc.12}
{
\sum_0^\infty \widetilde{\sigma }_j^2\lesssim h^{2m},
}
and the proposition follows from applying (\ref{sc.8}), (\ref{sc.10}),
(\ref{sc.11}) to (\ref{sc.9}) with $t=h^2$.
\end{proof}

\par We next review the choice of parameters for the random
perturbation in \cite{Sj2} (and \cite{Sj1}). This perturbation is of
the form $\delta Q_\omega $,
\ekv{sc.13}
{
Q_\omega =h^{N_1}q_\omega ,\ \delta =\tau_0h^{N_1+n},\ 0<\tau_0\le \sqrt{h},
}
where
\ekv{sc.14}
{
q_\omega (x)=\sum_{0<h\mu _k^0\le L}\alpha _k(\omega )\epsilon _k(x),\
\vert \alpha \vert_{{\bf C}^D}\le R,
}
and a possible choice of $L,R$ is 
\ekv{sc.15}
{
L=Ch^{-M},\quad R=Ch^{-\widetilde{M}},
}
with
\ekv{sc.16}
{
M=\frac{3n-\kappa }{s-\frac{n}{2}-\epsilon },\quad
\widetilde{M}=\frac{3n}{2}-\kappa +(\frac{n}{2}+\epsilon )M.
}
Here $\epsilon >0$ is any fixed parameter in $]0,s-\frac{n}{2}[$ and
$\kappa \in ]0,1]$ is the geometric exponent appearing in
(\ref{vo.2}), in our case equal to $1/2$.

The exponent $N_1$ is given by 
\ekv{sc.17}
{
N_1=\widetilde{M}+sM+\frac{n}{2},
}
and $q_\omega $ should be subject to a probability density on $B_{{\bf
  C}^D}(0,R)$ of the form $C(h)e^{\Phi (\alpha ;h)}L(d\alpha
)$, where
\ekv{sc.18}
{
\vert \nabla _\alpha \Phi \vert ={\cal O}(h^{-N_4}),
}
for some constant $N_4\ge 0$.

\par Write 
\ekv{sc.19}{q_\omega ^0=q_\omega ^1+q_\omega ^2,}
\ekv{sc.20}
{
q_\omega ^1=\sum_{0<h\mu _j^0\le L}\alpha _j^0(\omega )\epsilon _j,\
q_\omega ^2=\sum_{h\mu _j^0 > L}\alpha _j^0(\omega )\epsilon _j.
}
From Proposition \ref{sc1} and its proof, especially the observation
after (\ref{sc.9}), we know that
\ekv{sc.21}
{
\| h^mq_\omega ^2\|_{H^s}\le h\hbox{ with probability }\ge 1-\exp (C_0-\frac{1}{2Ch^{2(m-1)}}).
}
We write
$$
P+h^mq_\omega ^0=(P+h^mq_\omega ^2)+h^mq_\omega ^1,
$$
and recall that the main result in \cite{Sj2} is valid also when $P$
is replaced by the perturbation $P+h^mq_\omega ^2$, provided that $\|
h^mq_\omega ^2\|_{H^s}\le h$.

The next question is then wether $h^mq_\omega ^1$ can be written as
$\tau_0h^{2N_1+n}q_\omega $ where $q_\omega =\sum_{0<h\mu _j^0\le
  L}\alpha _j\epsilon _j$ and $\vert \alpha \vert_{{\bf C}^D}\le R$
with probability close to 1. We get 
$$
\alpha _j=\frac{1}{\tau_0}h^{m-2N_1-n}\alpha _j^0(\omega )\sim {\cal
  N}(0,\widehat{\sigma }_j^2),$$
$$ \frac{1}{\tau_0}h^{m-2N_1-n}(\mu
_j^0)^{-\rho }e^{-(\mu _j^0)^{\frac{\beta }{M+1}}}\lesssim
\widehat{\sigma }_j\lesssim
\frac{1}{\tau_0}h^{m-2N_1-n}(\mu _j^0)^{-\rho }.
$$

\par Applying (\ref{sc.9}), we get
\ekv{sc.22}
{
{\bf P}(\vert \alpha \vert_{{\bf C}^D}^2\ge R^2)\le \exp (C-\frac{R^2\tau_0^2}{Ch^{2(m-2N_1-n)}}),
}
which is ${\cal O}(1)\exp (-h^{-\delta })$ provided that 
\ekv{sc.23}
{
-2\widetilde{M}+2\frac{\ln (1/\tau _0)}{\ln (1/h)}+2(2N_1+n-m)
\le -\delta .
}
Here $\tau _0\le \sqrt{h}$ and if we choose $\tau _0=\sqrt{h}$ or more
generally bounded from below by some power of $h$, we see that
(\ref{sc.23}) holds for any fixed $\delta $, provided that $m$ is
sufficiently large.

\par In order to avoid such an extra assumption, we shall now
represent $\alpha _j^0$ for $h\mu _j^0\le L$ as the sum of two
independent Gaussian random variables. Let $j_0=j_0(h)$ be the largest $j$
for which $h\mu _j^0\le L$. Put 
\ekv{sc.24}
{
\sigma '=\frac{1}{C}h^Ke^{-Ch^{-\beta }},\hbox{ where }K\ge \rho
(M+1),\ C\gg 1
}
so that $\sigma '\le \frac{1}{2}\sigma _j$ for $1\le j\le j_0(h)$. The
factor $h^K$ is needed only when $\beta =0$.

\par For $j\le j_0$, we may assume that $\alpha _j^0(\omega )=\alpha
_j'(\omega )+\alpha _j''(\omega ),$ where $\alpha _j'\sim {\cal
  N}(0,(\sigma ')^2)$, $\alpha _j''\sim {\cal N}(0,(\sigma _j'')^2)$ are independent
random variables and 
$$
\sigma _j^2=(\sigma ')^2+(\sigma _j'')^2,
$$
so that 
$$
\sigma _j''=\sqrt{\sigma _j^2-(\sigma ')^2}\asymp \sigma _j .
$$

\par Put $q_\omega ^1 =q_\omega '+q_\omega ''$, where
$$
q_\omega '=\sum_{h\mu _j^0\le L}\alpha _j'(\omega )\epsilon _j,\
q_\omega ''=\sum_{h\mu _j^0\le L}\alpha _j''(\omega )\epsilon _j.
$$
Now (cf (\ref{sc.19})) we write
$$
P+h^mq_\omega ^0=(P+h^m(q_\omega ''+q_\omega^2))+h^mq_\omega '. 
$$
The main result of \cite{Sj2} is valid for random perturbations of
$$
P_0:=P+h^m(q_\omega ''+q_\omega ^2),
$$
provided that $\Vert h^m(q_\omega ''+q_\omega ^2)\Vert_{H^s}\le h$,
which again holds with a probability as in (\ref{sc.21}). The new random
perturbation is now $h^mq_\omega '$ which we write as $\tau
_0h^{2N_1+n}\widetilde{q}_\omega $, where $\widetilde{q}_\omega $
takes the form
\ekv{sc.25}{
\widetilde{q}_\omega (x)=\sum_{0<h\mu _j^0\le L}\beta _j(\omega
)\epsilon _j,
}
with new independent random variables 
\ekv{sc.26}
{
\beta _j=\frac{1}{\tau _0}h^{m-2N_1-n}\alpha _j'(\omega )\sim {\cal
  N}(0,(\frac{1}{\tau _0}h^{m-2N_1-n}\sigma '(h))^2).
}

\par Now, by (\ref{sc.9}),
$$
{\bf P}(\vert \beta \vert_{{\bf C}^D}^2>R^2)\le \exp ({\cal
  O}(1)D-\frac{R^2\tau _0^2}{{\cal O}(1)(h^{m-2N_1-n}\sigma '(h))^2}).
$$
Here by Weyl's law for the distribution of eigenvalues of elliptic
self-adjoint differential operators, we have $D\asymp (L/h)^n$. Moreover,
$L,R$ behave like certain powers of $h$. 
\begin{itemize}
\item In the case when $\beta =0$, we choose $\tau _0=h^{1/2}$.
Then for any $a>0$ we get 
$$
{\bf P}(\vert \beta \vert_{{\bf C}^D}>R)\le C\exp (-\frac{1}{Ch ^a})
$$  for any given fixed $a$, provided we choose $K$ large enough in 
(\ref{sc.24}).
\item In the case $\beta >0$ we get the same conclusion with $\tau
  _0=h^{-K}\sigma '$ if $K$ is large enough.
\end{itemize}

\par In both cases, we see that the independent random
variables $\beta _j$ in (\ref{sc.25}), (\ref{sc.26}) have a joint
probability density $C(h)e^{\Phi (\alpha ;h)}L(d\alpha )$, satisfying (\ref{sc.18}) for
some $N_4$ depending on $K$.

\par With $\kappa =1/2$, we put 
$$
\epsilon _0(h)=h^{\kappa } ((\ln \frac{1}{h})^2+\ln \frac{1}{\tau _0}),
$$
where $\tau _0$ is chosen as above. Notice that $\epsilon _0(h)$ is of
the order of magnitude $h^{\kappa -\beta }$ up to a power of $\ln
\frac{1}{h}$. Then Theorem 1.1 in \cite{Sj2} gives:
\begin{prop}\label{sc2}
There exists a constant $N_4>0$ depending on $\rho ,n,m$ such that the
following holds: Let $\Gamma \Subset \dot{{\bf C}}$ have piecewise
smooth boundary. Then $\exists C>0$ such that for $0<r\le 1/C$,
$\widetilde{\epsilon }\ge C\epsilon _0(h)$, we have with probability
\ekv{sc.27}
{
\ge 1-\frac{C\epsilon _0(h)}{rh^{n+\max
    (n(M+1),N_4+\widetilde{M})}}e^{-\frac{\widetilde{\epsilon
    }}{C\epsilon _0(h)}}-Ce^{-\frac{1}{Ch}},
}
that 
\eekv{sc.28}
{
\vert \# (h^mP^0_\omega )\cap \Gamma )-\frac{1}{(2\pi
  h)^n}\mathrm{vol\,}(p^{-1}(\Gamma ))\vert 
\le }
{\frac{C}{h^n}(\frac{\widetilde{\epsilon }}{r}+C(r+\ln (\frac{1}{r})\mathrm{vol\,}
(p^{-1}(\partial \Gamma +D(0,r))))).}
\end{prop}

As noted in \cite{Sj1} this gives Weyl asymptotics provided that 
\ekv{sc.29}
{
\ln (\frac{1}{r})\mathrm{vol\,}p^{-1}(\partial \Gamma +D(0,r))={\cal
  O}(r^\alpha ),
}
for some $\alpha \in ]0,1]$ (which would automatically be the case
if $\kappa $ had been larger than $1/2$ instead of being equal to
$1/2$), and we can then choose $r=\widetilde{\epsilon }^{1/(1+\alpha
  )}$, so that the right hand side of (\ref{sc.28}) becomes $\le
C\widetilde{\epsilon }^{\frac{\alpha }{1+\alpha }}h^{-n}$.

\par As in \cite{Sj1, Sj2} we also observe that if $\Gamma $ belongs
to a family ${\cal G}$ of domains satisfying the assumptions of the Proposition
uniformly, then with probability
\ekv{sc.30}
{
\ge 1-\frac{C\epsilon _0(h)}{r^2h^{n+\max
    (n(M+1),N_4+\widetilde{M})}}e^{-\frac{\widetilde{\epsilon
    }}{C\epsilon _0(h)}}-Ce^{-\frac{1}{Ch}},
}
the estimate (\ref{sc.28}) holds uniformly and simultaneously for all
$\Gamma \in {\cal G}$.

\section{End of the proof}\label{en}
\setcounter{equation}{0} 

Let $\theta _1,\theta _2,N_0$ be as in Theorem \ref{in1}, so that
$P(\theta _1,N_0)$ and $P(\theta _2,N_0)$ hold. Combining the
propositions \ref{vo1}, \ref{vo2}, \ref{vo3}, we
see that (\ref{sc.29}) holds for every $\alpha <1/N_0$ when
$\Gamma =\Gamma ^g_{\theta _1,\theta _2;1,\lambda }$, $\lambda >0$
fixed, and Proposition \ref{sc2} gives:
\begin{prop}\label{en1}
With the parameters as in Proposition \ref{sc2} and for every $\alpha
\in ]0,\frac{1}{N_0}[$, we have with probability
\ekv{en.1}
{
\ge 1 -\frac{C\epsilon _0(h)}{\widetilde{\epsilon }^{\frac{1}{1+\alpha
    }}h^{n+\max
    (n(M+1),N_4+\widetilde{M})}}e^{-\frac{\widetilde{\epsilon
    }}{C\epsilon _0(h)}}-Ce^{-\frac{1}{Ch}}
}
that
\ekv{en.2}
{
\vert \# (\sigma (h^mP_\omega )\cap \Gamma ^g_{\theta _1,\theta
  _2;1,\lambda })-\frac{1}{(2\pi h)^n}\mathrm{vol\,}(p^{-1}(\Gamma
^g_{\theta _1,\theta _2;1,\lambda }))\vert \le
C\frac{\widetilde{\epsilon }^{\frac{\alpha }{1+\alpha }}}{h^n}.
}
Moreover, the conclusion (\ref{en.2}) is valid simultaneously for all
$\lambda \in [1,2]$ and all $(\theta _1,\theta _2)$ in a set where
$P(\theta _1,N_0)$, $P(\theta _2,N_0)$ hold uniformly, with probability
\ekv{en.3}
{
\ge 1 -\frac{C\epsilon _0(h)}{\widetilde{\epsilon }^{\frac{2}{1+\alpha
    }}h^{n+\max
    (n(M+1),N_4+\widetilde{M})}}e^{-\frac{\widetilde{\epsilon
    }}{C\epsilon _0(h)}}-Ce^{-\frac{1}{Ch}}.
}
\end{prop}

\par For $0<\delta \ll 1$, choose $\widetilde{\epsilon
}=h^{-\delta }\epsilon _0\le Ch^{\frac{1}{2}-\beta -\delta }(\ln \frac{1}{h})^2$, so that
$\widetilde{\epsilon }/\epsilon _0=h^{-\delta }$. Then for some $N_5$
we have for every $\alpha \in ]0,1/N_0[$ that 
\ekv{en.4}
{
\vert \# (\sigma (h^mP_\omega )\cap \Gamma ^g_{\theta _1,\theta
  _2;1,\lambda })-\frac{1}{(2\pi h)^n}\mathrm{vol\,}(p^{-1}(\Gamma
^g_{\theta _1,\theta _2;1,\lambda }))\vert \le \frac{C_\alpha
}{h^n}(h^{\frac{1}{2}-\delta -\beta }(\ln \frac{1}{h})^2)^{\frac{\alpha
  }{1+\alpha }},
}
simultaneously for $1\le \lambda \le 2$ and all $(\theta _1,\theta _2)$ in a set where
$P(\theta _1,N_0)$, $P(\theta _2,N_0)$ hold uniformly, with probability
\ekv{en.5}
{
\ge 1-\frac{C}{h^{N_5}}e^{-\frac{1}{Ch^\delta }}.
}
Here $\alpha /(1+\alpha )\nearrow 1/(N_0+1)$ when $\alpha \nearrow
1/N_0$, so the upper bound in (\ref{en.4}) can be replaced by 
$$
\frac{C_\delta }{h^n}h^{(\frac{1}{2}-\beta -2\delta )/(N_0+1)}.
$$

Assuming 
$P(\theta _1,N_0)$, $P(\theta _2,N_0)$, we want to count
the number of eigenvalues of $P_\omega $ in 
$$
\Gamma _{1,\lambda }=\Gamma ^g_{\theta _1,\theta _2;1,\lambda }
$$
when $\lambda \to \infty $. Let $k(\lambda )$ be he largest integer
$k$ for which $2^k\le \lambda $ and decompose
$$
\Gamma _{1,\lambda }=(\bigcup_0^{k(\lambda )-1}\Gamma
_{2^k,2^{k+1}})\cup \Gamma _{2^{k(\lambda )},\lambda }.
$$
In order to count the eigenvalues of $P_\omega ^0$ in $\Gamma
_{2^k,2^{k+1}}$ we define $h$ by $h^m2^k=1$, $h=2^{-k/m}$, so that 
\begin{eqnarray*}
\# (\sigma (P_\omega ^0)\cap \Gamma _{2^k,2^{k+1}})&=&\# (\sigma
(h^mP^0_\omega )\cap \Gamma _{1,2}),\\ 
\frac{1}{(2\pi
  )^n}\mathrm{vol\,}(p^{-1}(\Gamma _{2^k,2^{k+1}}))&=&\frac{1}{(2\pi
 h)^n}\mathrm{vol\,}(p^{-1}(\Gamma _{1,2})).
\end{eqnarray*}
Thus, with probability $\ge 1-C2^{\frac{N_5k}{m}}e^{-2^\frac{\delta k}{m}/C}$ we have
\ekv{en.6}
{
\vert \# (\sigma (P_\omega ^0)\cap \Gamma _{2^k,2^{k+1}})
-\frac{1}{(2\pi )^n}\mathrm{vol\,}p^{-1}(\Gamma _{2^k,2^{k+1}})
\vert \le C_\delta 2^{\frac{kn}{m}}2^{-\frac{k}{m}(\frac{1}{2}-\beta -2\delta
)\frac{1}{N_0+1}}. 
}
Similarly, with probability $\ge 1-C2^{N_5k(\lambda )/m}
e^{-2^{\delta k(\lambda )/m}/C}$, we have
\ekv{en.7}
{
\vert \# (\sigma (P_\omega ^0)\cap \Gamma _{2^{k(\lambda
  )},\widetilde{\lambda }})
-\frac{1}{(2\pi )^n}\mathrm{vol\,}p^{-1}(\Gamma _{2^k(\lambda
  ),\widetilde{\lambda }})
\vert \le C_\delta \lambda ^{\frac{n}{m}}\lambda
^{-\frac{1}{m}(\frac{1}{2}-\beta -2\delta
)\frac{1}{N_0+1}}, 
}
simultaneously for all $\widetilde{\lambda }\in [\lambda ,2\lambda [$.

\par Now, we proceed as in \cite{Bo}, using essentially the Borel--Cantelli lemma. Use that
\begin{eqnarray*}
\sum_\ell^\infty 2^{N_5\frac{k}{m}}e^{-2^{\delta \frac{k}{m}}/C}&=&{\cal
  O}(1) 2^{N_5\frac{\ell}{m}}e^{-2^{\delta \frac{\ell}{m}}/C},\\
\sum_{2^k\le \lambda
}2^{k\frac{n}{m}}2^{-\frac{k}{m}(\frac{1}{2}-\beta -2\delta
  )\frac{1}{N_0+1}}&=&{\cal O}(1)\lambda
^{\frac{n}{m}-\frac{1}{m}(\frac{1}{2}-\beta -2\delta )\frac{1}{N_0+1}},
\end{eqnarray*}
to conclude that with probability $\ge
1-C2^{N_5\frac{\ell}{m}}e^{-2^{\delta \frac{\ell}{m}}/C}$, we have
$$
\vert \# (\sigma (P_\omega ^0)\cap \Gamma _{2^\ell ,\lambda })\vert
\le C_\delta \lambda ^{\frac{n}{m}-\frac{1}{m}(\frac{1}{2}-\beta -\delta
  )\frac{1}{N_0+1}}
+C(\omega )
$$
for all $\lambda \ge 2^\ell$. This statement implies Theorem
\ref{in1}. \hfill{$\Box$}

\begin{proofof} Theorem \ref{in2}. This is just a minor modification
  of the proof of Theorem \ref{in1}. Indeed, we already used the
  second part of Proposition \ref{sc2}, to get (\ref{en.7}) with the
  probability indicated there. In that estimate we are free to
  vary $(g,\theta _1,\theta _2)$ in ${\cal G}$ and the same holds for
  the estimate (\ref{en.6}). With these modifications, the same proof 
gives Theorem \ref{in2}.
\end{proofof}

\end{document}